\documentclass[11pt]{article}
\usepackage{latexsym, amscd, amsfonts, eucal, mathrsfs, amsmath, amssymb, amsthm, xypic,xr, stmaryrd, color, enumerate, tikz}
\usepackage{mathabx}
\usepackage{appendix}
\usepackage[all]{xy}
\usepackage{hyperref}
\usepackage{fullpage}
\newtheorem{theorem}{Theorem}[subsection]

\newtheorem{lemma}[theorem]{Lemma}
\newtheorem{proposition}[theorem]{Proposition}
\newtheorem{corollary}[theorem]{Corollary}

\theoremstyle{definition}
\newtheorem{definition}[theorem]{Definition}
\newtheorem{construction}[theorem]{Construction}
\newtheorem{convention}[theorem]{Convention}
\newtheorem{example}[theorem]{Example}

\newtheorem{warning}[theorem]{Warning}
\newtheorem{remark}[theorem]{Remark}

\usepackage{cleveref}
\usepackage{scrextend}

\DeclareMathOperator*{\hocolim}{hocolim}
\DeclareMathOperator*{\holim}{holim}
\DeclareMathOperator*{\colim}{colim}

\begin{document}
\title{Mod 2 power operations revisited}
\author{Dylan Wilson}
\maketitle
\begin{abstract} 
	In this mostly expository note we take advantage
	of homotopical and algebraic advances
	to give a modern
	account of power operations on the mod 2 homology
	of $\mathbb{E}_{\infty}$-ring spectra. The main
	advance is a quick proof of the Adem  
	relations utilizing the Tate-valued Frobenius as a homotopical
	incarnation of the total power operation. We also give a
	streamlined derivation of the action of power operations
	on the dual Steenrod algebra.
\end{abstract}
\newpage
\tableofcontents
\newpage

\section*{Introduction}\label{sec:intro}
\addcontentsline{toc}{section}{\nameref{sec:intro}}

As someone who entered college at about
the time that Netflix started automatically playing the next episode
of a series, I cannot imagine discovering or verifying the Adem relations
using the tools available to Adem \cite{adem}.\footnote{It was precisely
while trying and failing multiple times to prove the Adem relations
in equivariant homotopy theory that, in act of true laziness, I
stumbled upon the technique explained in this note.} I even find
it hard to \emph{remember} the Adem and Nishida relations. 

Luckily, there is a useful mnemonic device which utilizes
the \textbf{total power operation}:
	\[
	Q(t) := \sum_{i \in \mathbb{Z}} Q^i t^i.
	\]
Here $t$ is an indeterminate, 
and the operation $Q^i: A_* \to A_{*+i}$ acts on the homotopy of 
any $\mathbb{E}_{\infty}$-$\mathbb{F}_2$-algebra $A$.
The total power operation
then produces a map:
	\[
	Q(t): A_* \to A_*(\!(t)\!).
	\]
We extend $Q(t)$ to a ring map
	\[
	Q(t): A_*(\!(s)\!) \to A_*\llbracket s,t\rrbracket[s^{-1}, t^{-1}]
	\]
by requiring that
	\[
	Q(t)(s) = s + s^2t^{-1}.
	\]
With this convention, it is possible to restate
the Adem relations, following Bullett-Macdonald \cite{bullett-macdonald},
Steiner \cite{steiner}, and Bisson-Joyal \cite{bisson-joyal} as:
	\begin{itemize}
	\item (Adem relations) For any $x\in A_*$,
	$Q(t)Q(s)x$ is symmetric in
	$s$ and $t$. 
	\end{itemize}
The usual Adem relations are recovered using a trick with residues
which we will review in \S\ref{ssec:residues}. 
Steiner's proof that the above identity holds is to reduce it to
one of the expressions met in the proof of the Adem
relations as in \cite[p.119]{steenrod-epstein} and \cite[4.7(e,g,i)]{may}.

In the case of Steenrod operations acting on the cohomology
of a space $X$, there is a more conceptual argument
due to Segal \cite[\S4]{bullett-macdonald}.
One can use the diagonal map to produce
an version of the total power operation taking values in
$H^*(X \times \mathrm{B}\Sigma_2)$. Indeed, this is one of the
earlier constructions of Steenrod operations \cite[Ch. VII]{steenrod-epstein}.
The iterated total square then takes values in 
$H^*(X \times \mathrm{B}\Sigma_2 \times \mathrm{B}\Sigma_2)
= H^*(X)[s,t]$ but factors through the total fourth power
which takes values in $H^*(X \times \mathrm{B}\Sigma_4)$. The
automorphism swapping $s$ and $t$ arises as an inner auromorphism
of $\Sigma_4$ so the formula for the iterated square must be symmetric
in $s$ and $t$.
	
Our primary goal is to explain how the Tate diagonal
(\S\ref{ssec:diagonal}) on spectra allows for a similar argument
for general power operations. The reader could probably
reconstruct the argument themselves just from the observation
that the total power operation is the effect on homotopy of the
(non-$\mathbb{F}_2$-linear) map of spectra:
	\[
	A \stackrel{\Delta}{\longrightarrow} 
	(A \otimes_{\mathbb{F}_2}A)^{t\Sigma_2} \to A^{t\Sigma_2}.
	\]
In fact, we take this as a definition, and develop
all the basic properties of power operations efficiently from
there. We hope that this note will give a mnemonic
for the \emph{proofs} of the standard identities
for power operations in much the same way that the work 
of Steiner \cite{steiner}, Bisson-Joyal \cite{bisson-joyal},
and Baker \cite{baker} has provided mnemonics for their
\emph{statements}. 

\subsection*{Outline}

In \S\ref{sec:tate} and \S\ref{sec:powers}
we review the facts we need about the Tate
construction and the Tate diagonal, following Nikolaus-Scholze
\cite{nikolaus-scholze}. In \S\ref{sec:operations} we give three definitions of
the operations $Q^i$: the classical one, one due to Lurie
\cite[\S2.2]{DAGXIII}, and one in terms of the Tate valued Frobenius.
We then explain how to recover the first properties of power operations.

In \S\ref{sec:relations} we turn to the
Adem relations. The key thing to prove is that having a
$\Sigma_4$-equivariant map $A^{\otimes 4} \to A$ produces a
lift of the iterated total power operation through the Frobenius
$A \to A^{t\Sigma_4}$. This takes a little bit of work but the
reader could come up with the argument themselves if they
remember to use the universal property of the Tate diagonal
amongst natural transformations of exact, lax symmetric monoidal
functors over and over again. Indeed, this proof is an
excellent illustration of the computational utility of establishing
such universal properties in the first place.

Finally, in \S\ref{sec:steenrod}, we show how the Bisson-Joyal
and Baker formulations of the Nishida relations
arise naturally from the perspective of the
Tate-valued Frobenius. We end by explaining
how to recover Steinberger's formulas \cite[\S III.2]{bmms} for the action of
power operations on the dual Steenrod algebra. This last step
is mostly algebraic, and essentially due to Bisson-Joyal,
but we have included it for completeness.

\subsubsection*{Acknowledgements} The author is grateful to Tom Bachmann
for comments on an earlier draft.

\section{The Tate construction}\label{sec:tate}

We review the Tate construction (\S\ref{ssec:tate-definitions})
and its universal property (\S\ref{ssec:tate-monoidal})
as well as the important Warwick duality (\S\ref{ssec:warwick}) of Greenlees
\cite{greenlees2} which allows an
alternative computation of the Tate construction. We end 
(\S\ref{ssec:tate-example}) by spelling
out what happens in the case $G = \Sigma_2$. 

\subsection{Definitions}\label{ssec:tate-definitions}
Let $G$ be a finite group and $k$ an $\mathbb{E}_{\infty}$-ring,
and denote by 
	\[
	\mathsf{Mod}_k^{hG} := \mathsf{Psh}(\mathrm{B}G; \mathsf{Mod}_k)
	\]
the $\infty$-category of \textbf{Borel $G$-modules}. There is
a fully faithful embedding 
$\mathsf{Mod}_k^{hG} \to \mathsf{Mod}_k^{G}$ from Borel $G$-modules to
modules over $k$ in \emph{genuine} $G$-spectra
whose essential image consists
of the \textbf{Borel complete $G$-modules}, i.e. those
$X$ such that $X \to F(\mathrm{E}G_+, X)$ is
an equivalence. Let $\mathcal{F}$ be a collection of subgroups
closed under sub-conjugacy, and $\mathrm{E}\mathcal{F}$ the
$G$-space characterized up to homotopy by the requirement
	\[
	\mathrm{E}\mathcal{F}^{H} = \begin{cases}
	* & H \in \mathcal{F}\\
	\varnothing & H \notin \mathcal{F}
	\end{cases}
	\]
and define $\widetilde{\mathrm{E}\mathcal{F}}$ as the cofiber
of $\mathrm{E}\mathcal{F}_+ \to S^0$. 
Then the \textbf{$\mathcal{F}$-Tate spectrum} of a Borel
$G$-spectrum can be computed as \cite[p.443]{greenlees1}:
	\[
	X^{t\mathcal{F}} = (\widetilde{\mathrm{E}\mathcal{F}}
	\wedge F(\mathrm{E}G_+, X))^G,
	\]
where the right hand side is computed in genuine $G$-spectra.

It will be more convenient for us to think of the above as
a computation and not a definition. Instead, we opt to define
the Tate construction by a universal property, following
\cite{nikolaus-scholze}.

To that end, let
	\[
	\left(\mathsf{Mod}_k^{hG}\right)_{\mathcal{F}-\mathrm{ind}} \subseteq
	\mathsf{Mod}_k^{hG}
	\]
be the smallest full, stable subcategory containing all objects
which are left Kan extended from diagrams $\mathrm{B}H \to \mathsf{Sp}$
for some $H \in \mathcal{F}$. 

Recall \cite[\S I.3]{nikolaus-scholze} that, associated to any
exact functor $F: \mathsf{Mod}_k^{hG} \to \mathcal{E}$ to a
presentable stable $\infty$-category $\mathcal{E}$, there is a
natural transformation
	\[
	F \to L_{\mathcal{F}}F
	\]
which is initial amongst natural transformations to exact functors
which annihilate the subcategory 
$\left(\mathsf{Mod}_k^{hG}\right)_{\mathcal{F}-\mathrm{ind}}$. Concretely, 
$L_{\mathcal{F}}F$ is specified by the formula \cite[I.3.3]{nikolaus-scholze}:
	\[
	L_{\mathcal{F}}F(X) = \underset{
	\left(\mathsf{Mod}_k^{hG}\right)_{\mathcal{F}-\mathrm{ind}/X} 
	\ni Y}{\colim}
	F( \mathrm{cofib}(Y \to X)).
	\]
\begin{definition} With notation as above, we define
	\[
	(-)^{t\mathcal{F}} = L_{\mathcal{F}}((-)^{hG}):
	\mathsf{Mod}_k^{hG} \to \mathsf{Mod}_k.
	\]
More generally, if $G \subseteq G'$ we define
	\[
	(-)^{t\mathcal{F}} = L_{\mathcal{F}}((-)^{hG}):
	\mathsf{Mod}_k^{hG'} \to \mathsf{Mod}_k^{hW_{G'}G}
	\]
where $W_{G'}G = N_{G'}G/G$ is the Weyl group of $G$ in $G'$.
\end{definition}

\begin{example} When $\mathcal{F}$ consists only of the trivial
subgroup, we denote $X^{t\mathcal{F}} = X^{tG}$. This can be
computed as the cofiber of the trace map $X_{hG} \to X^{hG}$.
\end{example}

\begin{example} Suppose
$G \subseteq \Sigma_n$ is a subgroup and let $\mathcal{F} = \mathcal{T}$,
be the family of subgroups of $G$ which do \emph{not} act transitively
on $\{1, ..., n\}$. When $G = C_n$ this coincides with the more
commonly seen family of proper subgroups; and when $G = C_p$
this coincides with the family consisting of only the trivial subgroup.
\end{example}

\subsection{Warwick duality}\label{ssec:warwick}

We can dualize the construction in the previous section and
define the \textbf{opposite $\mathcal{F}$-Tate spectrum}\footnote{We stole this
name from \cite{glasman-lawson}.} as
	\[
	X^{t^{\mathrm{op}}\mathcal{F}}:=
	\underset{\left(\left(\mathsf{Mod}_k^{hG}
	\right)_{\mathcal{F}-\mathrm{ind}}\right)_{X/}\ni Y}
	{\holim}
	\mathrm{fib}(X \to Y)_{hG}
	\]
	
Greenlees proved \cite[\S B]{greenlees2} that this construction is not
really new:

\begin{theorem}[Warwick duality] There is a canonical equivalence
	\[
	X^{t^{op}\mathcal{F}} \simeq \Sigma^{-1}X^{t\mathcal{F}}.
	\]
\end{theorem}

In particular, we obtain extra functoriality: if 
$\mathcal{F} \subseteq \mathcal{F}'$, then the original construction
produces a canonical map $(-)^{t\mathcal{F}'} \to (-)^{t\mathcal{F}}$
while the opposite construction, composed with suspension,
produces a map $(-)^{t\mathcal{F}} \to (-)^{t\mathcal{F}'}$.

\subsection{Monoidal structure}\label{ssec:tate-monoidal}

We will make much use of the following excellent
description of the lax symmetric 
monoidal structure on the Tate construction. 

\begin{proposition} There is a natural transformation of lax symmetric
monoidal functors \[(-)^{hG} \to (-)^{t\mathcal{F}}\] which is initial
amongst natural transformations of lax symmetric monoidal functors
with target an exact functor that annihilates
$\left(\mathsf{Mod}_k^{hG}\right)_{\mathcal{F}-\mathrm{ind}}$.
\end{proposition}

This follows
from the more general result \cite[I.3.6]{nikolaus-scholze}
about the relationship between Verdier quotients and lax
symmetric monoidal structures.

\subsection{An example}\label{ssec:tate-example}

Let $k$ be a field of characteristic 2. Then we have
	\[
	\pi_*k^{h\Sigma_2} \simeq
	\mathrm{H}^{-*}(\mathrm{B}\Sigma_2, k)= k \llbracket t\rrbracket
	\]
where $t \in \pi_{-1}k^{h\Sigma_2}$ is the 
Stiefel-Whitney class of the canonical line bundle. The Tate
construction has the effect of inverting $t$ and we can compute
	\[
	\pi_*k^{t\Sigma_2} = k(\!(t)\!),
	\]
the algebra of Laurent series over $k$. 

On the other side, the homotopy orbits $k_{h\Sigma_2}$ have a dual
basis on homotopy
	\[
	\pi_*k_{h\Sigma_2} = k\{e_0, e_1, ...\}
	\]
where $e_i$ is the linear dual of $t^i$. The trace map
	\[
	k_{h\Sigma_2} \to k^{h\Sigma_2}
	\]
is zero on homotopy groups and so we have a short exact sequence
	\[
	0\to k\llbracket t\rrbracket \to k(\!(t)\!) \to
	\pi_*\Sigma k_{h\Sigma_2} \to 0
	\]
which identifies the last term as the quotient $k(\!(t)\!)/ k\llbracket t\rrbracket$.
This provides another basis for the homotopy of $k_{h\Sigma_2}$,
and the two are related by the correspondence
	\[
	e_i \leftrightarrow t^{-i-1}.
	\]
Under this interpretation, the composite map
	\[
	k^{t\Sigma_2} \to \Sigma k_{h\Sigma_2} \to \Sigma k
	\]
is given by sending a Laurent series $g(t) = \sum a_it^i$ to the
\textbf{residue} $a_{-1}$. 

Finally, Warwick duality in this context translates to the
computation \cite[16.1]{greenlees-may}
	\[
	\Sigma^{-1}k^{t\Sigma_2} =
	\holim_n (\Sigma^{-n\tau}k)_{h\Sigma_2}=
	\holim_n k \wedge \left(\mathbb{R}P^{\infty}\right)^{-n\tau}
	=
	\holim_n k \wedge \mathbb{R}P^{\infty}_{-n},
	\]
where $\tau$ is the sign representation.

\section{Tate powers}\label{sec:powers}

The source of power operations is the symmetry present on
$X^{\otimes n}$. In \S\ref{ssec:powers} we review several
constructions based on this symmetry. In \S\ref{ssec:goodwillie}
we explain how the construction $X \mapsto 
(X^{\otimes n})^{t\mathcal{T}}$ arises as a Goodwillie derivative;
in particular this construction is exact. In \S\ref{ssec:diagonal},
following
\cite{nikolaus-scholze}, we describe the spectral analog of the diagonal
map we will use when defining power operations.

\subsection{Variants of extended powers}\label{ssec:powers}

Let $\mathcal{C}$ be a symmetric monoidal $\infty$-category. Then
there is a natural functor
	\[
	\mathcal{C} \to \mathcal{C}^{h\Sigma_n} = 
	\mathsf{Fun}(\mathrm{B}\Sigma_n, \mathcal{C})
	\]
given as the composite
	\[
	\mathcal{C} \stackrel{\delta}{\to}
	\left(\mathcal{C}^{\times n}\right)^{h\Sigma_n}
	\to \mathcal{C}^{h\Sigma_n}
	\]
where the latter map is a choice of tensor product. In other words,
for every $X \in \mathcal{C}$, the object $X^{\otimes n}$ has
a $\Sigma_n$-action.

If $\mathcal{C}$ admits homotopy limits and colimits, we can form
both a `symmetric' power of an object and a `divided' power of an object.
We do this more generally for a fixed subgroup $G \subseteq \Sigma_n$.

\begin{definition} We define symmetric and divided power functors as:
	\[
	\mathrm{Sym}^G(X) := \left(X^{\otimes n}\right)_{hG},
	\quad
	\Gamma^G(X) := \left(X^{\otimes n}\right)^{hG}.
	\]
\end{definition}

Finally, if $\mathcal{C}=\mathsf{Mod}_k$ is the $\infty$-category
of $k$-modules over an $\mathbb{E}_{\infty}$-ring $k$, then:

\begin{definition} Let $G \subseteq \Sigma_n$ be a subgroup.
We define the \textbf{Tate power} of $X$ as
	\[
	\mathrm{T}_G(X) := (X^{\otimes n})^{t\mathcal{T}}
	\]
where $\mathcal{T}$ is the family of non-transitive subgroups
of $G$.
\end{definition}

In each case we abbreviate $G$ as $n$ if $G = \Sigma_n$.

\subsection{Tate powers as a Goodwillie derivative}\label{ssec:goodwillie}

Let $\mathcal{C}$ and $\mathcal{D}$ be stable, presentable
$\infty$-categories. Then the full subcategory
	\[
	\mathsf{Fun}^{\mathrm{ex}}(\mathcal{C}, \mathcal{D})
	\subseteq \mathsf{Fun}(\mathcal{C}, \mathcal{D})
	\]
admits a left adjoint \cite[6.1.1.10]{HA}, the 1-excisive approximation:
	\[
	P_1: \mathsf{Fun}(\mathcal{C}, \mathcal{D})
	\to \mathsf{Fun}^{\mathrm{ex}}(\mathcal{C}, \mathcal{D}).
	\]
In the case where $F(0)=0$, we may compute $P_1F$ as
\cite[6.1.1.23,6.1.1.27]{HA}:
	\[
	P_1F(X) = \hocolim_n \Omega^n_{\mathcal{D}}
	F(\Sigma^n_{\mathcal{C}}X).
	\]
I believe the following is well-known but do not know a reference.

\begin{proposition} With notation as in \S\ref{ssec:powers},
There is an equivalence:
	\[
	P_1\Gamma^G \simeq T_G.
	\]
\end{proposition}
\begin{proof} Let $V$ denote the standard
representation of $\Sigma_n$ on $\mathbb{R}^n$
and $\overline{V}$ the reduced standard representation.
By the formula above we have
	\begin{align*}
	P_1\mathrm{\Gamma}^G(X) &=
	\hocolim_j \Omega^j\mathrm{\Gamma}^G(\Sigma^jX)
	\\&\simeq \hocolim \Omega^j(\Sigma^{jV}X^{\otimes n})^{h\Sigma_n}
	\\&\simeq \hocolim (\Sigma^{j\overline{V}}X^{\otimes n})^{h\Sigma_n}
	\\&\simeq \hocolim 
	(S^{j\overline{V}}\wedge F(\mathrm{E}G_+, X^{\otimes n}))^G\\
	&\simeq 
	(S^{\infty\overline{V}} \wedge F(\mathrm{E}G_+, X^{\otimes n}))^G.
	\end{align*}
The last identification used that genuine fixed points commute with all
homotopy limits and colimits. Finally, observe that $S^{\infty\overline{V}}$
is a model for $\widetilde{\mathrm{E}\mathcal{T}}$.
\end{proof}

The same argument computes the Goodwillie \emph{co}derivative
of $\mathrm{Sym}^G$:

\begin{proposition} The Goodwillie coderivative of $\mathrm{Sym}^G$
is $((-)^{\otimes n})^{t^{op}\mathcal{T}} = \Sigma^{-1}T_G$.
\end{proposition}

This last observation motivates the excellent
account of stable power operations
given by Glasman-Lawson \cite{glasman-lawson}.

\subsection{The Tate diagonal}\label{ssec:diagonal}

Recall the following result of Nikolaus \cite[Cor. 6.9]{nikolaus}:

\begin{proposition}\label{prop:initial} The forgetful functor
$U: \mathsf{Mod}_k \to \mathsf{Sp}$ is initial amongst exact, lax
symmetric monoidal functors to spectra.
\end{proposition}

In the previous section we identified $T_G$ as a Goodwillie derivative.
In particular, $T_G$ is exact. It also has a lax symmetric monoidal
structure, being a composite of lax symmetric monoidal functors.
So we get the following:

\begin{corollary} There is an essentially unique natural transformation
of lax symmetric monoidal functors $U \to UT_G$.
\end{corollary}

We refer to this map $\Delta_G: M \to T_G(M)$ as the \textbf{Tate diagonal}.

\begin{remark} This is not the same as the Tate diagonal in
\cite{nikolaus-scholze} unless $k=S^0$,
since we use the tensor product in
$\mathsf{Mod}_k$. Of course there is an evident relationship between
the two: the Tate diagonal above is just the composite
	\[
	M \to (M^{\wedge n})^{t\mathcal{T}} \to (M^{\otimes n})^{t\mathcal{T}}.
	\]
\end{remark}
\begin{warning} The Tate diagonal is \emph{not} $k$-linear.
\end{warning}

\section{Power operations}\label{sec:operations}

We now fix a field $k$
of characteristic 2 and let $\mathsf{Mod}_k$ be the
$\infty$-category of $k$-module (spectra). In \S\ref{ssec:3defns}
we serve up power operations three ways, and then verify they
agree in \S\ref{ssec:compare}. In between we verify the first
properties of power operations up to the Cartan formula.
We emphasize that this section does not show off the utility
of the approach via the Tate-valued Frobenius, but we have
included the proofs since they are still pleasant.

\subsection{Three definitions of operations}\label{ssec:3defns}

First we specify the objects on which power operations will act.

\begin{definition} We say that $A \in \mathsf{Mod}_k$ is
\textbf{equipped with a symmetric multiplication} if we have
specified a map $\mathrm{Sym}^2(A) \to A$ of $k$-modules.
Equivalently, if we have specified a map $A^{\otimes 2} \to A$
in $\mathsf{Mod}_k^{h\Sigma_2}$.
\end{definition}

\begin{remark} A $k$-module with a symmetric multiplication
is the same as an object of $\mathscr{C}(2,\infty)$
in the notation of \cite{may}. 
\end{remark}

To give the classical construction of power operations we'll need a computation.

\begin{lemma} For any integer $n$ there is a canonical equivalence
	\[
	\mathrm{Sym}^2(\Sigma^nk) \simeq \Sigma^{2n} k_{h\Sigma_2}.
	\]
\end{lemma}
\begin{proof} The object $(\Sigma^nk)^{\otimes 2} = \Sigma^{nV}k$
in $\mathsf{Mod}_k^{h\Sigma_2}$ corresponds to a map
$\mathrm{B}\Sigma_2 \to \mathsf{Mod}_k^{h\Sigma_2}$ which is
determined by a map 
	\[
	\Sigma_2 \to \mathsf{End}_k(\Sigma^{2n}k)
	\simeq \mathsf{End}_k(k,k) \simeq k.
	\]
of $\mathbb{E}_1$-monoids. The map
factors through the units $k^{\times}$, but $k$ has characteristic
2 and hence no nontrivial square roots of unity. So the action is
trivial and the result follows.
\end{proof}

The following construction is the current standard
definition of power operations.

\begin{construction}[Hands-on power operations]\label{cstr:hands-on}
Let $A$ be
a $k$-module equipped with a symmetric multiplication. Given
$x \in \pi_nA$ and $i\ge n$, define $Q^i(x) \in \pi_{n+i}A$ as
the composite
	\[
	\xymatrix{
	S^{n+i} \ar[rr]^-{\Sigma^{2n}e_{i-n}} && 
	\Sigma^{2n}k_{h\Sigma_2} \simeq
	\mathrm{Sym}^2(\Sigma^nk)
	\ar[rr]^-{\mathrm{Sym}^2(x)}&&
	\mathrm{Sym}^2(A) \ar[r]& A
	}.
	\]
\end{construction}

This has the benefit of generalizing well to power operations
for other cohomology theories, but in the case of mod 2 cohomology
there is a more uniform option. The author learned this next approach
from \cite[\S 2.2]{DAGXIII} and has not found an earlier reference,
but a more recent and detailed account can be found in \cite{glasman-lawson}.

First we need a preliminary observation.
Let $T'_2: \mathsf{Mod}_k \to \mathsf{Mod}_k$ denote the left
Kan extension of the restriction of $T_2$ to the full subcategory
of compact objects. This endomorphism commutes with all colimits
and so (\cite[7.1.2.4]{HA}) there is a bimodule $B$ and an
equivalence $T'_2(M) \simeq B \otimes M$. By evaluating on
$M=k$ we deduce that $B=k^{t\Sigma_2}$ as a left $k$-module.
Notice, by construction, we have a natural map $B \otimes M \to T_2(M)$.

\begin{construction}[Stable power operations]\label{cstr:stable-opns}
Let $A$ be a $k$-module equipped with a symmetric multiplication.
The element $t^{-i-i} \in \pi_{i+1}k^{t\Sigma_2}$ extends to a
\emph{right} module map $\Sigma^{i}k \to \Sigma^{-1}B$.
We now define $Q^i: \Sigma^iA \to A$ as the (non $k$-linear!) composite:
	\[
	\Sigma^iA=\Sigma^ik\otimes A \to
	\Sigma^{-1}B\otimes A \to \Sigma^{-1}T_2(A)
	\to \mathrm{Sym}^2(A) \to A.
	\]
\end{construction}

This construction emphasizes the role of $\Sigma^{-1}k^{t\Sigma_2}$
as acting on $A$, but we can also record this information in
a kind of co-action. For that we first need a computation.

\begin{lemma} For any $k$-module $M$ equipped
with the trivial $\Sigma_2$-action, there is a canonical
equivalence of $\pi_*k^{t\Sigma_2}$-modules
	\[
	\pi_*M^{t\Sigma_2} \simeq M_*(\!(t)\!).
	\]
\end{lemma}
\begin{proof} It suffices to prove $\pi_*M^{h\Sigma_2} \simeq
M_*\llbracket t\rrbracket$. From the skeletal filtration on 
$\mathrm{B}\Sigma_2$ we have
	\[
	M^{h\Sigma_2} \simeq \holim F(\mathrm{sk}_j\mathrm{B}\Sigma_2,
	k) \otimes M
	\]
	and $\pi_*F(\mathrm{sk}_j\mathrm{B}\Sigma_2,
	k) \otimes M = M_*[t]/t^{j+1}$. The transition maps are surjective
	so there is no $\lim^1$ term in the Milnor exact sequence and
	the result follows.
\end{proof}

\begin{construction}[Tate-valued Frobenius]\label{cstr:frobenius}
Let $A$ be a $k$-module equipped with a symmetric multiplication.
Define the \textbf{total power operation} as the composite:
	\[
	Q(t): A \stackrel{\Delta_2}{\longrightarrow} T_2(A)
	= (A^{\otimes 2})^{t\Sigma_2} \to A^{t\Sigma_2}.
	\]
We then define $Q^i : A \to \Sigma^{-i}A$ as the composite
	\[
	A \to A^{t\Sigma_2} \stackrel{t^{-i-1}}{\to} \Sigma^{-i-1}A^{t\Sigma_2}
	\to \Sigma^{-i} A_{h\Sigma_2} \to \Sigma^{-i}A.
	\] 
\end{construction}

In \S\ref{ssec:compare} we will verify that the two definitions
of the endomorphism $Q^i: \Sigma^iA \to A$ coincide and that
each induce the operation $Q^i: \pi_nA \to \pi_{n+i}A$ on homotopy.
For now we will assume this compatibility.

\begin{remark}[Naturality of Frobenius] \label{rmk:naturality} The Tate-valued
Frobenius can be defined for any \emph{spectrum} equipped
with a symmetric multiplication, as the composite
$A \to (A \wedge A)^{t\Sigma_2} \to A^{t\Sigma_2}$. Since
the $k$-module Tate diagonal factors through the spectrum Tate diagonal,
we learn that the Tate valued Frobenius only depends on the underlying
$\mathbb{E}_{\infty}$-ring. In particular, the Tate-valued Frobenius
is natural for maps $A \to B$ of $\mathbb{E}_{\infty}$-rings, independent
of any compatibility with $k$-module structures.
\end{remark}

\subsection{First properties}\label{ssec:properties}

The first properties follow easily from the Tate-valued Frobenius
description, with the exception of the squaring property, which
is most readily seen through the classical definition.

\begin{proposition} The operations $Q^i$ satisfy the following
properties.
	\begin{enumerate}[\upshape(i)]
	\item (Additivity) $Q^i(x+y) = Q^i(x)+Q^i(y)$.
	\item (Suspension) $\Omega Q^i(x) = Q^i(\Omega x)$.
	\item (Squaring) $Q^{|x|}(x) = x^2$.
	\item (Instability) $Q^i(x) = 0$ if $i<|x|$. 
	\item (Action on cohomology) If $A = F(X, k)$
	where $X$ is a pointed space, then $Q^i(x) = 0$
	for $i>0$ and $Q^0(x) = x$.
	\end{enumerate}
\end{proposition}
\begin{proof}
	\begin{enumerate}[(i)]
	\item (Additivity) Since $Q^i$ is induced by
	a map of spectra, it is automatically additive.
	\item (Suspension) The Tate diagonal is
	a natural transformation of exact functors, so
	$\Delta_{2, \Omega A} \simeq \Omega \Delta_2$.
	Exactness of $T_2$ then ensures that
	$\Omega T_2(A) \to T_2(\Omega A)$ is an equivalence,
	and composing with the multiplication on
	$\Omega A$ identifies $\Omega Q(t)$ with the total
	power operation for $\Omega A$, which was to be shown.
	\item (Squaring) Using Construction \ref{cstr:hands-on},
	observe that $Q^{|x|}(x)$ is the image of the bottom
	class in $\mathrm{Sym}^2(\Sigma^nk)$, which is the left
	vertical arrow in the diagram:
		\[
		\xymatrix{
		\Sigma^nk \otimes \Sigma^nk\ar[d]\ar[r]^{x\otimes x} &
		M \otimes M \ar[d]\\
		\mathrm{Sym}^2(\Sigma^nk)\ar[r] & \mathrm{Sym}^2(M)
		}
		\]
	The result follows by chasing the diagram clockwise.
	\item (Instability) By (ii) we may replace $A$ by $\Omega^{|x|-i}A$
	and thereby reduce to the case that $A= \Omega B$ and
	$i = |x|$. By (iii), $Q^ix = x^2$, but the multiplication on
	$\Omega B$ is always trivial, since $S^1 \to S^1 \wedge S^1$
	is null.
	\item (Action on cohomology) By naturality we may replace
	$X$ with $K(k, n)$ and $x$ with the fundamental class.
	Now the result follows for degree reasons.
	\end{enumerate}
\end{proof}

\subsection{Cartan formula}\label{ssec:cartan}

If $A$ and $A'$ are equipped with symmetric multiplications
then $A \otimes A'$ inherits a canonical symmetric multiplication
as well. In this case we have an external Cartan formula:

\begin{proposition}[Cartan formula] 
	\[
	Q(t)(x \otimes y)
	= Q(t)(x) \otimes Q(t)(y) \in (A \otimes A')(\!(t)\!).
	\]
\end{proposition}
\begin{proof} The formula is equivalent to commutativity of the square:
	\[
	\xymatrix{
	A\otimes A'\ar@{=}[d]\ar[r]& T_2(A) \otimes T_2(A') \ar[r]\ar[d] &
	A^{t\Sigma_2} \otimes A^{t\Sigma_2}\ar[d]\\
	A\otimes A' \ar[r]& T_2(A \otimes A') \ar[r] & (A \otimes A')^{t\Sigma_2}
	}
	\]
	The left square commutes because the Tate diagonal is
	a transformation of lax symmetric monoidal functors. The
	right hand square commutes by naturality of the lax structure
	map \[(-)^{t\Sigma_2} \otimes (-)^{t\Sigma_2}
	\to (-\otimes-)^{t\Sigma_2}\] applied to
	$(A\otimes A')^{\otimes 2}\simeq A^{\otimes 2} \otimes
	A^{'\otimes 2} \to A \otimes A'$.
\end{proof}

\begin{corollary} $Q^n(x\otimes y) = \sum_{i+j=n} Q^i(x)\otimes Q^j(y)$.
\end{corollary}

As a corollary of the proof, we see:

\begin{corollary} If $A \otimes A \to A$ is a map of objects equipped
with symmetric multiplications, then $Q(t): A \to A^{t\Sigma_2}$
is also a map of objects equipped with symmetric multiplications.
\end{corollary}

\subsection{An example}\label{ssec:opns-example}

We revisit our example $k^{t\Sigma_2}$, but to avoid confusion
we change the name of the generator: $k^{t\Sigma_2}_*=
k(\!(s)\!)$. From the equivalence 
$k^{h\Sigma_2} = F(\mathrm{B}\Sigma_{2+}, k)$ together
with properties (iii), (iv), and (v), we see that
	\[
	Q(t)(s) = s + s^2t^{-1}.
	\]
The Cartan formula now determines the behavior of $Q(t)$
in general:
	\[
	Q(t)\sum_i a_i s^i = \sum_i a_i (s+s^2t^{-1})^i.
	\]

\subsection{Comparing the definitions}\label{ssec:compare}

Let $B$ denote the bimodule from Construction \ref{cstr:stable-opns},
which is equivalent to $k^{t\Sigma_2}$ as a left $k$-module.
Let $k \to B$ extend $1 \in \pi_0k^{t\Sigma_2}$ as a right module map.

\begin{lemma} The composite
	\[
	A \to B \otimes A \to T_2(A)
	\]
above is equivalent to the Tate diagonal $\Delta_2$.
\end{lemma}
\begin{proof} Indeed, first observe that
by the universal property of spectra \cite[1.4.2.23]{HA}, we have
	\[
	\Omega^{\infty}:
	\mathsf{Fun}^{\mathrm{ex}}(\mathsf{Mod}_k, \mathsf{Sp})
	\stackrel{\simeq}{\longrightarrow}
	\mathsf{Fun}^{\mathrm{lex}}(\mathsf{Mod}_k,
	\mathsf{Spaces}).
	\]
Now let $U: \mathsf{Mod}_k \to \mathsf{Sp}$ be the forgetful
functor. Then $\Omega^{\infty}U$ is corepresented by $k$, so the
Yoneda lemma applied to the previous observation implies that
	\[
	\mathrm{Map}_{\mathsf{Fun}^{\mathrm{ex}}(\mathsf{Mod}_k, 
	\mathsf{Sp})}(U, UT_2) \simeq \Omega^{\infty}k^{t\Sigma_2}.
	\]
Since the Tate diagonal is a transformation of lax symmetric monoidal
functors, the
transformation $U \to UT_2$ evaluates on $k$ to the unit
$k \to k^{t\Sigma_2}$. Combining this with the previous observation
we learn that the Tate diagonal is the unique transformation
$U \to UT$ which corresponds to the element $1 \in \pi_0k^{t\Sigma_2}$.
This completes the proof.
\end{proof}

Thus the map 
	\[
	A \to B\otimes A \to T_2(A) \to A^{t\Sigma_2}
	\]
coincides with the Tate valued Frobenius. Now observe that
the last three terms are left modules over $k^{t\Sigma_2}$,
so multiplication by $t^{-i-1}$ and
naturality of $(-)^{t\Sigma_2} \to \Sigma (-)_{h\Sigma_2}$
gives a commutative diagram:
	\[
	\xymatrix{
	A\ar[r]\ar[dr]
	&B\otimes A \ar[r] \ar[d]& T_2(A)\ar[r]\ar[d]& A^{t\Sigma_2}\ar[d]\\
	&\Sigma^{-i-1}B\otimes A \ar[r]& \Sigma^{-i-1}T_2(A)\ar[r]\ar[d]
	&\Sigma^{-i-1}A^{t\Sigma_2}\ar[d]\\
	&&\Sigma^{-i}\mathrm{Sym}^2(A)\ar[r] & \Sigma^{-i}A
	}
	\]
Chasing the diagram around clockwise gives the definition of $Q^i$
in terms of the total power operation. Chasing the diagram around
clockwise gives the definition of $Q^i$ in terms of Construction
\ref{cstr:stable-opns}. So these two constructions agree.

Now we compare with the classical construction. The equivalence
$(\Sigma^{n}k)^{\otimes 2} \simeq \Sigma^{2n}k$ in
$\mathsf{Mod}_k^{h\Sigma_2}$ gives a commutative diagram
	\[
	\xymatrix{
	\Sigma^{-1}T_2(\Sigma^nk) \ar[r]\ar[d]^{\simeq} & 
	\mathrm{Sym}^2(\Sigma^nk)\ar[d]^{\simeq}\\
	\Sigma^{2n-1}k^{t\Sigma_2} \ar[r] & \Sigma^{2n} k_{h\Sigma_2}
	}
	\]
Since the bottom horizontal map is surjective on homotopy, so is the top,
and we see that $\Sigma^{2n}e_{i-n}$ on the lower right corresponds
to $t^{-i-1}y$ on the top left, where $y \in \pi_{n}\Sigma^nk$ is the generator.
Now let $x: S^{n} \to A$ be a class and form the diagram:
	\[
	\xymatrix{
	S^{i+n}\ar[rr]^{t^{-i-1}y}\ar[drr]_{t^{-i-1}x}
	&&\Sigma^{-1}T_2(\Sigma^nk) \ar[r] & 
	\mathrm{Sym}^2(\Sigma^nk)\ar[d]\\
	&&\Sigma^{-1}T_2(A) \ar[r] & \mathrm{Sym}^2(A)
	}
	\]
Traversing clockwise gives $Q^i(x)$ as in Construction \ref{cstr:hands-on}
and traversing counterclockwise gives the image of $x$ under $Q^i$
as in Construction \ref{cstr:stable-opns}, and this completes the argument.

\section{Adem relations}\label{sec:relations}

The Adem relations arise from relating the iterated
total power operation to a total fourth power operation.
In \S\ref{ssec:iterate} we first explain how to lift the
iterated total power operation to an intermediate Tate spectrum.
In \S\ref{ssec:adem-objects} we show that the existence
of extra symmetry on iterated multiplication allows us to factor
further through a total fourth power operation. This implies
a version of the Adem relations as an identity between formal
Laurent series in two variables, and in \S\ref{ssec:residues}
we essentially perform the maneuver from \cite{bullett-macdonald}
to recover the usual Adem relations.

For notational ease we adopt the following convention
in this section:

\begin{convention} If $G \subseteq \Sigma_n$ is a subgroup,
and $\mathcal{T}$ denotes the family of non-transitive subgroups of $G$,
then we denote $(-)^{t\mathcal{T}}$ by $(-)^{\tau G}$. 
\end{convention}

\subsection{Iterated power operations}\label{ssec:iterate}

Suppose $A$ is a $k$-module equipped with a symmetric
multiplication. Iterating the multiplication gives a map
	\[
	A^{\otimes 4} \to A
	\]
which need not admit an $\Sigma_4$-equivariant structure. However,
it can be made $\Sigma_2 \wr \Sigma_2$-equivariant, so we may
define a map:
	\[
	A \to T_{\Sigma_2\wr\Sigma_2}(A) \to A^{\tau \Sigma_2\wr\Sigma_2}.
	\]
Our first goal is to show that this lifts the iterated total power operation.
\begin{proposition}\label{prop:iterate} Let $A$ be a $k$-module equipped with 
a symmetric multiplication. Then there is a canonical commutative
diagram:
	\[
	\xymatrix{
	&&A^{\tau\Sigma_2\wr\Sigma_2}\ar[d]\\
	A \ar[rr]_-{Q(t) \circ Q(s)}\ar[urr] && (A^{t\Sigma_2})^{t\Sigma_2}
	}
	\]
\end{proposition}
\begin{proof} First consider the following diagram:
	\[
	\xymatrix{
	T_2(A) \ar[r]\ar[d] & T_2(T_2(A)) \ar[r]\ar[d] & 
	((A^{\otimes 4})^{t\Sigma_2})^{t\Sigma_2}\ar[d]\\
	A^{t\Sigma_2} \ar[r] & T_2(A^{t\Sigma_2})\ar[r]&
	(A^{t\Sigma_2})^{t\Sigma_2}
	}
	\]
The first square commutes by naturality of the Tate diagonal
applied to the map $T_2(A) \to A$. The second square commutes
by naturality of the lax structure map for $(-)^{t\Sigma_2}$.

It follows that $Q(t) \circ Q(s)$ can be written as the composite:
	\[
	A \to T_2(T_2(A)) \to ((A^{\otimes 4})^{t\Sigma_2})^{t\Sigma_2}
	\to (A^{t\Sigma_2})^{t\Sigma_2}.
	\]
Now consider both $(-)^{\tau\Sigma_2\wr\Sigma_2}$ and
$((-)^{t\Sigma_2})^{t\Sigma_2}$ as exact functors
$\mathsf{Mod}_k^{h\Sigma_4} \to \mathsf{Mod}_k$. We have
a natural transformation
	\[
	(-)^{h\Sigma_2\wr\Sigma_2} \to (-)^{h\Sigma_2 \times\Sigma_2}
	= ((-)^{h\Sigma_2})^{h\Sigma_2} \to 
	((-)^{t\Sigma_2})^{t\Sigma_2}
	\]
where the first map is induced by the inclusion
$\Sigma_2 \times \Sigma_2 \to (\Sigma_2\times\Sigma_2) \rtimes \Sigma_2
= \Sigma_2\wr\Sigma_2$ given by the diagonal on the first factor.
By the universal property of the Tate construction (\S\ref{ssec:tate-definitions}),
we get a natural transformation $(-)^{\tau\Sigma_2\wr\Sigma_2}
\to ((-)^{t\Sigma_2})^{t\Sigma_2}$. In particular, applied to the
multiplication map $A^{\otimes 4} \to A$, we get a commutative diagram:
	\[
	\xymatrix{
	T_{\Sigma_2\wr\Sigma_2}(A) \ar[r]\ar[d] 
	& A^{t\Sigma_2\wr\Sigma_2}\ar[d]\\
	((A^{\otimes 4})^{t\Sigma_2})^{t\Sigma_2}\ar[r] &
	(A^{t\Sigma_2})^{t\Sigma_2}
	}
	\]
Finally, the composite
	\[
	\Gamma^{\Sigma_2\wr\Sigma_2}
	\to \Gamma^{\Sigma_2\times \Sigma_2} \simeq 
	\Gamma^2 \circ \Gamma^2 \to T_2 \circ T_2
	\]
yields a natural transformation $T_{\Sigma_2\wr\Sigma_2} \to
T_2 \circ T_2$ from the universal property of $T_{\Sigma_2\wr\Sigma_2}$
as the Goodwillie derivative of $\Gamma^{\Sigma_2\wr\Sigma_2}$. 
The diagram
	\[
	\xymatrix{
	T_{\Sigma_2\wr\Sigma_2}\ar[d]\ar[dr]\\
	T_2 \circ T_2 \ar[r]& (((-)^{\otimes 4})^{t\Sigma_2})^{t\Sigma_2}
	}
	\]
commutes by the same universal property, and the result follows.
\end{proof}

\subsection{Adem objects}\label{ssec:adem-objects}

For the Adem relations to hold we need the symmetric multiplication
to satisfy an extra condition.

\begin{definition} We say that $k$-module $A$ equipped
with a symmetric multiplication is an 
\textbf{Adem object} if there exists a map
$\mathrm{Sym}^4(A) \to A$ such that the diagram
	\[
	\xymatrix{
	\mathrm{Sym}^2(\mathrm{Sym}^2(A))\ar[d]\ar[r] &
	\mathrm{Sym}^2(A) \ar[d]\\
	\mathrm{Sym}^4(A) \ar[r] & A
	}
	\]
commutes up to homotopy.
\end{definition}

\begin{proposition}\label{prop:adem}
If $A$ is an Adem object, then we have a commutative
diagram:
	\[
	\xymatrix{
	&&A^{\tau \Sigma_4}\ar[d]\\
	&&A^{\tau\Sigma_2 \wr \Sigma_2}\ar[d]\\
	A \ar[rr]_{Q(t)\circ Q(s)}\ar[urr]\ar[uurr] && (A^{t\Sigma_2})^{t\Sigma_2}
	}
	\]
\end{proposition}
\begin{proof} By Proposition \ref{prop:iterate}, the bottom triangle commutes.
The top triangle commutes because each arrow is a transformation
of exact, lax symmetric monoidal functors, and $U:
\mathsf{Mod}_k \to \mathsf{Sp}$ is initial amongst such functors
(Proposition \ref{prop:initial}). 
\end{proof}

\begin{theorem}[Adem relations] If $A$ is an Adem object and
$x \in \pi_*A$ is an element, then
$Q(t)(Q(s)x)$ is symmetric in the variables $s$ and $t$. Explicitly:
	\[
	\sum_{i,j} (Q^iQ^jx)(s+s^2t^{-1})^jt^i
	= \sum_{i,j} (Q^iQ^jx)(t+t^2s^{-1})^js^i.
	\]
\end{theorem}
\begin{proof} By Proposition \ref{prop:adem}, the iterated total power
operation factors through $A^{\tau\Sigma_4}$ and the operation
which swaps $s$ and $t$ arises from an inner automorphism of
$\Sigma_4$ which thus acts trivially on the Tate constructs, whence the
claim. The explicit formula follows from the basic properties of power
operations, the Cartan formula,
and the computation in \S\ref{ssec:opns-example}.
\end{proof}

\subsection{Residues and relations}\label{ssec:residues}

Now we recall how to recover the individual Adem relations using
the power series identity above. 

\begin{proposition} Let $A$ be an Adem object and $x \in A_*$
a homotopy class. Then:
	\[
	Q^iQ^j(x) = \sum_{\ell} \binom{\ell-j-1}{2\ell-i} Q^{i+j-\ell}Q^{\ell}(x).
	\]
\end{proposition}
\begin{proof} In the previous section we showed
	\[
	\sum_j Q(t)(Q^jx) (s+s^2t^{-1})^j =
	\sum_{i, j} (Q^kQ^jx)(t+t^2s^{-1})^j s^k.
	\]
Let $u = s+s^2t^{-1}$ and observe that this is composition invertible
as a power series in $s$ with coefficients in $k(\!(t)\!)$. Now,
	\[
	Q(t)(Q^jx) = \sum_i (Q^iQ^jx)t^i
	\]
is the coefficient of $u^j$ on the left hand side, so we would like
to compute the coefficient of $u^j$ on the right hand side.
It will be convenient to reindex the right hand side, for fixed $j$, as:
	\[
	\sum_{i, \ell} (Q^{i+j-\ell}Q^{\ell}x) (t+t^2s^{-1})^{\ell}s^{i+j-1}.
	\]
Observe that $du = ds$ since $2=0$ in $k$, and hence
	\[
	\mathrm{res}(u^{-j-1}(Q^{i+j-\ell}Q^{\ell}x) (t+t^2s^{-1})^{\ell}s^{i+j-1} du)
	=
	\mathrm{res}(u^{-j-1}(Q^{i+j-\ell}Q^{\ell}x) (t+t^2s^{-1})^{\ell}s^{i+j-1} ds).
	\]
Fixing $i$ and $\ell$ and 
writing $u = st^{-1}(t+s)$ and $(t+t^2s^{-1}) = s^{-1}t(t+s)$, we have
	\[
	u^{-j-1}(t+t^2s^{-1})^{\ell}s^{i+j-1} = t^{\ell+j+1}s^{i-2\ell-1}(t+s)^{\ell-j-1}.
	\]
The coefficient of $s^{-1}$ in the previous expression is then
	\[
	\binom{\ell-j-1}{2\ell-i}t^i
	\]
and the result follows.
\end{proof}

\section{Relationship to the Steenrod algebra}\label{sec:steenrod}

In this section we restrict to the case $k=\mathbb{F}_2$ for ease
of exposition. In \S\ref{ssec:coact-on-tate}
we recall the Steenrod coaction on the Tate spectrum, then in
\S\ref{ssec:nishida} we use this to give a succinct proof of
the Nishida relations. Finally, in \S\ref{ssec:steenrod} we show
how this determines the action of $Q(t)$ on the dual
Steenrod algebra, following an idea of Bisson-Joyal.

\subsection{Coaction on the Tate spectrum}\label{ssec:coact-on-tate}

The map $k = k \wedge S^0 \to k \wedge k$ gives rise to
a map $k^{t\Sigma_2} \to (k \wedge k)^{t\Sigma_2}$ if we equip
the source and target with trivial $\Sigma_2$ action.

This induces a completed coaction:
	\[
	\psi_R: k(\!(t)\!) \to \mathcal{A}_*(\!(t)\!).
	\]
Now recall that Milnor defined generators\footnote{Note that we are following
Milnor's convention and not the more recent trend of using
$\zeta_i$ to denote the \emph{conjugates} of Milnor's generators.}
of the dual Steenrod algebra by the identity
	\[
	\psi_R(t) = \sum \zeta_i t^{2^i}.
	\]
\subsection{Nishida relations}\label{ssec:nishida}

The easier version of the Nishida relations in this context
is in terms of the coaction.

\begin{theorem}[Bisson-Joyal, Baker] Let $X$ be a
spectrum equipped with an equivariant symmetric multiplication
$X^{\wedge 2}_{h\Sigma_2} \to X$. Then 
	\[
	\sum_i \psi_R(Q^ix) t^i = Q(\overline{\zeta}(t))\psi_R(x)
	\in (\mathrm{H}_*X \otimes \mathcal{A}_*)(\!(t)\!).
	\]
\end{theorem}
\begin{proof} Let $k \wedge X$. Then the right coaction
$k \wedge X \to (k \wedge X) \otimes_k (k\wedge k)$
is a map of \emph{spectra} equipped
with symmetric multiplications (though it is not a map of $k$-modules
equipped with symmetric multiplications). By Remark \ref{rmk:naturality}
this yields a commutative diagram:
	\[
	\xymatrix{
	k \wedge X \ar[r]^-{\psi_R}\ar[d] & (k \wedge X) \otimes_k (k \wedge k)\ar[d]\\
	(k \wedge X)^{t\Sigma_2} \ar[r]_-{(\psi_R)^{t\Sigma_2}} & 
	((k \wedge X) \otimes_k (k \wedge k))^{t\Sigma_2}
	}
	\]
The bottom map arises by applying $(-)^{t\Sigma_2}$ to 
$k \wedge X \to k \wedge X \wedge k$ and this precisely gives the
completed coaction on $(k \wedge X)^{t\Sigma_2}$. In other words:
	\[
	\psi_R(Q(t)x) = Q(t)(\psi_R(x)).
	\] 
Since $\psi_R$ is a ring map, and $\psi_R(t) = \zeta(t)$,
this becomes:
	\[
	\sum \psi_R(Q^ix) \zeta(t)^i = Q(t)(\psi_R(x)).
	\]
Now substitute the conjugate series $\overline{\zeta}(t)$ for $t$
and use the defining relation $\zeta(\overline{\zeta}(t)) = t$.
\end{proof}

\subsection{Action on the dual Steenrod algebra}\label{ssec:steenrod}

The following description of the action of the $Q^i$ on $\mathcal{A}_*$
is essentially
that of Bisson-Joyal \cite[\S1, Prop. 6]{bisson-joyal}.

\begin{theorem}[Bisson-Joyal]\label{thm:action}
The total power operation on the Milnor generators $\zeta_i$
is determined implicitly by the identity:
	\begin{align}
	\zeta(s) + \zeta(s)^2\zeta(t)^{-1} 
	&= \sum_i (Q(t)\zeta_i) (s^{2^i} + s^{2^{i+1}}t^{-2^i})
	\end{align}
	\begin{align}
	t^{2^n}Q(t)\zeta_n &= \left(\sum_{i\ge n+1}\zeta_it^{2^i}\right)
	+ \zeta(t)^{-1}\left(\sum_{i\ge n}\zeta_i^2t^{2^{i+1}}\right).
	\end{align}
\end{theorem}
\begin{proof} Write $\pi_*k^{h\Sigma_2} = k\llbracket s\rrbracket$. Then:
	\[
	\psi_R(Q(t)s) = Q(t)\psi_R(s).
	\]
Now use the identities $Q(t)s = s+s^2t^{-1}$ and $\psi_R(s) = \zeta(s)$.
Comparing coefficients for $s^{2^n}$ gives a recursion for
$Q(t)\zeta_n$ starting with $Q(t)\zeta_0 = Q(t)1 = 1$ and (2)
is solves the recursion.
\end{proof}

It is not difficult to extract the earlier results of Steinberger
\cite[\S III.2]{bmms}.

\begin{corollary}[Steinberger] For $i\ge 2$,
$Q^{2^i-2}\zeta_1 = \overline{\zeta}_i$.
\end{corollary}
\begin{proof} From Theorem \ref{thm:action}(2) above in the case
$n=1$ we get
	\[
	Q(t)\zeta_1 = t^{-1}+\zeta_1 + \zeta(t)^{-1}.
	\]
So, for $i\ge 2$, change of variables and a quick computation gives:
	\begin{align*}
	Q^{2^i-2}\zeta_1 &= \mathrm{res}(t^{-2^i+1}\zeta(t)^{-1}dt)\\
	&=\mathrm{res}(\overline{\zeta}(u)^{-2^i +1}u^{-1}du) = \overline{\zeta}_i.
	\end{align*}
\end{proof}

\begin{corollary}[Steinberger] We have
$Q^{2^i}\zeta_i = \zeta_{i+1} + \zeta_i^2\zeta_1$ and
$Q^{2^i}\overline{\zeta}_i = \overline{\zeta}_{i+1}$.
\end{corollary}
\begin{proof} The case $i=0$ is evident, so assume $i\ge 1$.
The coefficient of $t^0s^{2^{i+1}}$ on the right
hand side of Theorem \ref{thm:action}(1) is visibly
$Q^{2^i}\zeta_i + Q^0(\zeta_i)=Q^{2^i}\zeta_i$. The constant
term of $\zeta(t)^{-1}$ is $\zeta_1$, so the coefficient of
$t^0s^{2^{i+1}}$ on the left hand side is 
$\zeta_{i+1}+\zeta_i^2\zeta_1$. The other identity follows
from this one by induction and the defining relation for conjugation.
\end{proof}

\bibliographystyle{amsalpha}
\nocite{*}
\bibliography{Bibliography}

\end{document}